\documentclass[11pt]{amsart}
\usepackage{latexsym, amsmath, amssymb,amsthm,amsopn,amsfonts}
\usepackage{version}
\usepackage{epsfig,graphics,color,graphicx,graphpap}
\usepackage{amssymb}
\usepackage{todonotes}
\usepackage{listings}
\usepackage[framemethod=tikz]{mdframed}

\usepackage[citecolor=blue]{hyperref}
\usepackage[framemethod=tikz]{mdframed}

\numberwithin{equation}{section}
\newtheorem{theorem}{Theorem}[section]
\newtheorem{corollary}[theorem]{Corollary}

\newtheorem{lemma}[theorem]{Lemma}
\newtheorem{definition}{Definition}[section]
\newtheorem{hypothesis}[]{Hypothesis}
\newtheorem{remark}{Remark}[section]

\newcommand{\LC}{\left(}
\newcommand{\RC}{\right)}

\newcommand{\p}{\partial}

\DeclareMathOperator{\supp}{supp}
\DeclareMathOperator{\diam}{diam}

\newcommand{\R}{\mathbb R}

\title[ ]{Parameter Reconstruction for general transport equation}
\author[Lai]{Ru-Yu Lai}
\address{School of Mathematics, University of Minnesota, Minneapolis, MN 55455, USA}
\curraddr{}
\email{rylai@umn.edu }
\author[Li]{Qin Li}
\address{Department of Mathematics, University of Wisconsin-Madison, Madison, WI 53705, USA}
\curraddr{}
\email{qinli@math.wisc.edu}
\thanks{}

\thanks{\noindent\textbf{Key words.} Inverse problem, general transport equation, stability estimate, Carleman estimate.\\
\thanks{\textbf{AMS subject classifications.} 35R30, 65L09}}

\begin{document}
\begin{abstract}
We consider the inverse problem for the general transport equation with external field, source term and absorption coefficient. We show that the source and the absorption coefficients can be uniquely reconstructed from the boundary measurement, in a Lipschitz stable manner. Specifically, the uniqueness and stability are obtained by using the Carleman estimate in which a special weight function is designed to pick up information on the desired parameter.

\end{abstract}

\maketitle
\tableofcontents
 


\section{Introduction}
Kinetic theory is a full body of theory that characterizes the behavior of a large number of particles that follow the same physical laws. In particular, the generalized transport equation is one classical model in kinetic theory. Let $u(t,x,v)$ denote the density of particles at time $t$ and position $x$ with velocity $v$. The generalized transport equation characterizes the evolution of $u$, and in a general form, the equation read as
\begin{align}\label{general transport}
\p_t u +v\cdot\nabla_x u+\textbf{E}\cdot\nabla_v u=-q u+S\,\quad \hbox{for }\ (t,x,v)\in(0,T)\times \Omega\times \R^3\,,
\end{align}
where $\Omega$ is an open bounded and convex domain in $\R^3$ with smooth boundary $\p\Omega$. 
The three terms on the left of \eqref{general transport} characterize the trajectory of particles which satisfies the ODE system
\[
\dot{x} = v\,,\quad\dot{v} = \textbf{E}(x)\,.
\]
Here $\textbf{E}(x)$ is the external force. For electrons in semi-conductors, for example, $\textbf{E}(x)$ can be regarded as the electric field. In some situations, $\textbf{E}$ is a self-consistent field generated by charged particles, and the field is induced through the Poisson equation. The two terms on the right of \eqref{general transport} are the damping term and the source term, respectively. We term $q = q(x,v)$ the absorption coefficient which reflects the rate of particles being absorbed by the media, and $S=S(t,x,v)$ the source term which reflects the rate of new particles introduced into the domain from the external source. In some cases, $q$ and $S$ are also functionals of $u$, possibly making the full equation nonlinear. In this paper, we take $\textbf{E}(x)$ as a deterministic given function that only depends on $x$, and moreover, we take $q$ and $S$ as functions independent of $u$. We note that with zero source and zero damping, the equation is known as the Vlasov equation which is the collision-less Boltzmann equation.


The equation \eqref{general transport} has a unique solution when the initial condition
\begin{equation}\label{eqn:initial}
u(t=0,x,v)= g(x,v)\, 
\end{equation}
and the boundary condition are suitably imposed. For the incoming flow, the boundary condition is imposed at the incoming coordinates. More specifically, we define
\begin{align}\label{gamma+}
\Gamma_{\pm}=\{(x,v)\in \p \Omega\times \R^3: \pm \ n(x)\cdot v > 0\}\,,
\end{align}
where $n(x)$ is the unit outer normal to $\p\Omega$ at the point $x\in \p\Omega$. This means $\Gamma_-$ collects all coordinates on the boundary with the corresponding velocity pointing into the domain, while $\Gamma_+$ collects the outgoing particle coordinates. Then the incoming flow is imposed on $\Gamma_-$, that is,
\begin{equation}\label{eqn:boundary}
u|_{\Gamma_-} = h(t,x,v)\,.
\end{equation}
The boundary and initial conditions are required to be compatible, which means
\[
g|_{\Gamma_-} = h(0,x,v)\,.
\]

The forward problem for the transport equation is to find the solution $u(t,x,v)$ to \eqref{general transport} with the initial and boundary conditions given in~\eqref{eqn:initial} and~\eqref{eqn:boundary}. On the other hand, for the inverse problem for the transport equation \eqref{general transport}, one seeks to recover the unknown terms from some measurable data. More precisely, suppose that the absorption coefficient $q$ or the source term $S$ is unknown, can one adjust the initial and boundary condition, and utilize certain measurements taken on the boundary to uniquely and stably determine $q$ or $S$?
In particular, one widely used boundary measurement is the \textit{albedo} operator which is defined by
\[
\mathcal{A}:\; (h,g) \rightarrow u|_{\Gamma_+}\,.
\]
Then one seeks for good choice of $(h,g)$ to trigger the information embedded in $\mathcal{A}$ for reconstructing $q$ and $S$.

The inverse transport problem finds its wide applications in optical imaging, remote sensing, semi-conductor designing, to name a few. On the numerical and engineering sides, many algorithms that are rooted in Baye's theory, or optimization have been developed and extensively studied. These methods provide some reconstructions of the parameters using the data collected in experiments. To some extent, one can argue these methods provide the best reconstruction based on the provided information numerically. However, the theoretical studies are largely in lack. Even the most basic question, is the information from the albedo operator truly enough to uniquely and stably identify the parameters, is unknown. The only exception is the study of the radiative transfer equation (RTE), a classical kinetic model for photon particles. Since photon particles do not accelerate/decelerate, in the RTE model, one sets $\textbf{E}=0$, and this makes the trajectory of particles much easier to analyze. The uniqueness results are found in~\cite{CS1, CS2, CS3, CS98, SU2d} and stability estimates have been derived in \cite{Bal14, Bal10, Bal18, Wang1999}. Moreover, the transport equation in the diffusion scaling are studied in \cite{LLU18, ZhaoZ18}. The major technique used in most of these results is the singular decomposition of the Schwartz kernel for the albedo operators $\mathcal{A}$, developed in~\cite{CS2,  CS98, Stefanov_2003}. Based on this technique, one can decompose the measurements into multiple components according to their different types of singularities and then these different components are used in various ways for the reconstruction. Interested readers are referred to some nicely written review papers \cite{Arridge1999, Bal_transport, Kuireview, Stefanov_2003} for this particular kind of inverse kinetic models.


The main goal of this paper is to fill the theoretical gap for a larger class of kinetic equations, the general transport equation \eqref{general transport}, where non-trivial external force $\textbf{E}$ presents. We study in this scenario, if and how either $q$ or $S$ can be reconstructed. The major difference between this problem and the widely investigated RTE is that our non-trivial $\textbf{E}$ significantly complicates the trajectory of particles, which makes the classical singular decomposition technique invalid. Therefore, we rely on the Carleman estimate which will be discussed in Section \ref{sec:Carleman}.

The Carleman estimate is a technique initiated by Bukhgeim and Klibanov in~\cite{BK1981}. It is an important tool for proving uniqueness and stability in reconstructing coefficients in partial differential equations, especially transport type equations, as seen in~\cite{Gaitan14, Yamamoto2016, KlibanovP2006, Klibanov08, Machida14}. The applications of Carleman estimates to inverse problems for hyperbolic systems are largely summarized in~\cite{Yamamotobook}. One key feature of the Carleman estimate is that, depending on a particular equation being investigated, some special weight functions are designed to extract the desired property from the estimate. These weights, when multiplied on the original equation, enlarge certain parts of information in the solution while suppressing the rest, and if strong enough, the deviation in the coefficients can be upper bounded by the deviation in the measurements, leading to the uniqueness and the Lipschitz stability. It is a rather general strategy and permits the recovery with one single measurement, and thus serves as a powerful tool in inverse problem, especially for wave and non-stationary transport type of equations. We want to point out that in \cite{Yamamoto2016}, the inverse transport problem with a variable velocity was studied, and the setup is relatively similar to the setting we have here.

\subsection{Main results} 
We will be mainly working on the $L^2$ space, and to unify the notation, we denote
\[
\|F\|^2_{L^2(\Omega\times\R^3)}:= \int_{\R^3} \int_\Omega |F|^2dx dv\,,\quad\|F\|^2_{L^2([0,T]\times\Gamma_+)}:= \int^T_0\int_{\Gamma_+} |F|^2d\sigma dt,
\]
where $d\sigma=|n(x)\cdot v| d\mu(x)dv$ is the surface measure and $d\mu(x)$ is the measure on $\partial\Omega$.

Throughout the paper, let $\mathcal{P}_\mathcal{U}$ denote
\begin{align}\label{definition of Pu}
    \mathcal{P}_\mathcal{U}=\{f:\ \|f(0,x,v)\|_{L^\infty( \Omega \times \mathcal{U})}\leq C_1,\ \|\p_t f\|&_{L^\infty([0,T]\times  \Omega \times \mathcal{U})}\leq C_1, \notag\\
    & \ \hbox{and } \inf_{\Omega \times \mathcal{U}}|f(0,x,v)|\geq C_2\}
\end{align}
for some given constants $C_1$ and $C_2>0$, where $\mathcal{U}\subset\R^3$ is a subset in the velocity space. We also denote the support of a function $q$ by $\supp(q)$.

%
%
 
The reconstruction of $q$ and $S$ are summarized in the following two separate theorems.


\subsubsection{Reconstruction of absorption coefficient}
We first assume that the external field $\textbf{E}$ and the source $S$ are known. Then the inverse problem is to reconstruct $q$ from the measurement on $\Gamma_+$. In particular, we want to show that the difference in $q$ would be visible from this boundary measurement. 

Let $g$ be the initial condition and $h$ be the incoming function on $\Gamma_-$. Suppose that $u_{j}$ are the solutions to the problem
\begin{align}\label{boltzmann4}
\left\{ \begin{array}{ll}
\p_t u_j+v\cdot\nabla_x u_j+\textbf{E}\cdot\nabla_v u_j+q_j u_j=S  & \hbox{in }(0,T)\times \Omega\times\R^3,\\
u_j(0,x,v) = g(x,v) &\hbox{in }\Omega\times \R^3,\\
u_j = h &\hbox{in }(0,T)\times \Gamma_-,\\
\end{array}\right. 
\end{align}
with absorption coefficients $q_j$ for $j=1,2$, respectively. Then we have the following theorem for estimating the discrepancy in $q$.
\begin{theorem}\label{main thm stability estimate} 
Let $\textbf{E}$, $S$, $q_j$, $h$, and $g$ satisfy certain assumptions (to be specified in Theorem \ref{thm:wellposedness}). 
Suppose $u_2\in\mathcal{P}_\mathcal{U}$, and $q_j\in L^\infty(\Omega\times\R^3)$ have compact supports in $v$ such that 
\[
   \supp (q_1-q_2)(x,\cdot)\subset \mathcal{U}\ \ \hbox{for any }x\in \Omega \,.
\]
Then there exist positive constants $c$ and $C$ depending on $\Omega,\ \textbf{E}$, $S$, $T$, $\mathcal{U}$, and $C_j$ defined in \eqref{definition of Pu} such that
 
\begin{align}\label{introduction estimate}
	c\|\p_tu_1-\p_t u_2\|_{L^2([0,T]\times\Gamma_+)} \leq  \|q_1-q_2\|_{L^2(\Omega\times\R^3)}\leq C\|\p_tu_1-\p_t u_2\|_{L^2([0,T]\times\Gamma_+)}\,.
\end{align}
\end{theorem}

This theorem indicates that we not only have uniqueness in the reconstruction of $q$, but also obtain the Lipschitz stability as shown in \eqref{introduction estimate}.

We would like to note that in Theorem~\ref{main thm stability estimate} we utilize the nontrivial initial data $g$, the incoming data and the outgoing data. When we take $\p_t u_1=\p_t u_2$ on $[0,T]\times\Gamma_+$, it implies that anisotropic absorption coefficients $q_1(x,v)=q_2(x,v)$ from \eqref{introduction estimate}.  
While in general, for the anisotropic media with $q$ depending on $v$, the unique determination of $q$ is not always valid by relying solely on the albedo operator $\mathcal{A}$ without the information on initial data. To see this, let's consider the case $\mathbf{E}=S=0$ in \eqref{general transport}. 
Thus, the most one can recover from the albedo operator $\mathcal{A}:u|_{\Gamma_-}\rightarrow u|_{\Gamma_+}$ are the integrals
\begin{align}\label{line integral}
	\int_{\R}q(x+sv,v)ds,
\end{align}
see \cite{CS2, CS98,Stefanov_Tamasan}. This is not sufficient to determine $q(x,v)$ since one can always change the $x$-variable in the direction $v$ in \eqref{line integral}, the integral still preserves the same value. 
Moreover, in \cite{Stefanov_Tamasan}, it was shown that such anisotropic absorption coefficient can only be recovered up to a gauge transformation, namely,
$$
    q_2(x,v)=q_1(x,v) - v\cdot\nabla_x\log\phi,
$$
with $\phi=1$ on $\p\Omega$ provided that their albedo operators are identical for both $q_1$ and $q_2$ in the stationary RTE problem. This implies the non-uniqueness of $q$. Even for the time-dependent transport equation in \cite{CS2}, this non-unique result for $q$ was also observed when the albedo operator on the boundary is the only given data.
In our setting, however, we additionally assume that the nontrivial $g$ is given, which prevents the occurrence of the gauge transformation in the transport equation \eqref{general transport}. Hence, the uniqueness of $q(x,v)$ is valid and, moreover, the stability estimate holds by using a single measurement. 

\subsubsection{Reconstruction of source}
The other scenario we consider is the reconstruction of $S$ when $\textbf{E}$ and $q$ are known. Let $u_j$, $j=1,2$, be the solution to the problem
\begin{align}\label{boltzmann diff source}
\left\{ \begin{array}{ll}
\p_t u_j+v\cdot\nabla_x u_j+\textbf{E}\cdot\nabla_v u_j+q  u_j=S_j & \hbox{in }(0,T)\times \Omega\times \R^3\,, \\
u_j(0,x,v)= g(x,v) &\hbox{in }\Omega\times \R^3\,, \\
u_j = h &\hbox{in }(0,T)\times \Gamma_-\,, \\
\end{array}\right. 
\end{align} 
with the source $S_j$, respectively.
Following a similar argument as in the proof of Theorem \ref{main thm stability estimate}, the source term can also be reconstructed from the given boundary data $\p_tu$ on $\Gamma_+$, as stated in the following theorem.

\begin{theorem}\label{thm: source}
Under certain conditions on $\textbf{E}$, $S_j$, $q$, $h$, and $g$ (to be specified in Theorem \ref{thm:wellposedness}), suppose that the sources are of the form $$S_j(t,x,v)=\widetilde S_j(x,v)S_0(t,x,v)\,,$$ where $S_0\in\mathcal{P}_\mathcal{U}$ and  
\[
\supp(\widetilde{S}_1-\widetilde{S}_2)(x,\cdot)\subset \mathcal{U} \ \ \hbox{for any }x\in \Omega \,.
\]
Then one has
\begin{align}\label{diff source introduction estimate}
c\|\p_tu_1-\p_t u_2\|_{L^2([0,T]\times\Gamma_+)} \leq \|\widetilde S_1-\widetilde S_2\|_{L^2(\Omega\times\R^3)}\leq C \|\p_tu_1-\p_t u_2\|_{L^2([0,T]\times\Gamma_+)} \,,
\end{align}
where $c$ and $C$ are positive constants depending on $\Omega,\ \textbf{E}$, $q$, $T$, $\mathcal{U}$, and $C_j$ defined in \eqref{definition of Pu}.  
\end{theorem}

Similarly, the second inequality in~\eqref{diff source introduction estimate} indicates the uniqueness and the Lipschitz stability in the reconstruction of $\widetilde{S}$.

We should emphasize that in this paper we only determine one unknown at a time while assuming the others are known. In practice, however, it is more interesting to simultaneously reconstruct both functions $q$ and $S$ based on the boundary measurement on $\Gamma_+$, but the task is beyond what we study in this paper. A major difficulty there is that it is not clear if the Carleman estimate could be extended to treat the situation. We leave this problem for future studies. 

To make our approach more clear, we briefly summarize the strategy for the derivation of stability estimates \eqref{introduction estimate} and \eqref{diff source introduction estimate} in the following steps:
\begin{enumerate}
	\item[1.] We first establish the energy estimate for initial boundary value problem for the equation \eqref{general transport}, and it will be stated in Lemma \ref{lemma energy}. 
	\item[2.] We then design a suitable weight function $\phi$ for the Carleman estimate to carry through, and the result will be stated in Lemma \ref{Carleman estimate}. This estimate plays a key role in controlling the solution on parts of the boundary. 
	\item[3.] We then introduce a smooth cut-off function $\chi$ in time $t$ and a smooth cut-off function $\psi$ in velocity $v$ to pick up the information in $\p_t u$. 
	\item[4.] We derive the equation for $\chi\psi\p_t u$ and then apply the Carleman estimate on this equation. When the initial condition is trivial, meaning $u(0,x,v)=0$, the initial flux $\p_t u(0,x,v)$ contains the inhomogeneous term ($q$ or $S$) of the transport equation. Such result is stated in Lemma~\ref{main result 1}, where some assumption on the solution is imposed, and in Lemma~\ref{main result 2} we eliminate this assumptions by incorporating the energy estimate from Lemma~\ref{lemma energy}.
	\item[5.] These results are finally used on the equation for $u_1-u_2$ to reconstruct $q_1-q_2$ or $\widetilde{S}_1-\widetilde{S}_2$, which leads to the two main theorems above. Their proof can be found in Section \ref{proof of thm}.
\end{enumerate}

\subsection{Outline} The rest of the paper is organized as follows. We first summarize in Section \ref{sec: well posedness} the well-posedness for the forward problem \eqref{transport_equation} and introduce some notations. In Section \ref{sec:Carleman}, we derive the energy estimate and the Carleman estimate for the transport equation \eqref{general transport}. In Section \ref{proof of thm}, these estimates are used to treat the solution with cut-off functions in the proof of the main theorems. Numerical evidences are presented in Section \ref{sec:numerics}, and the numerical results confirm the linear dependence of the discrepancy in the measurement and the discrepancy in the coefficient.

\section{Preliminary Results and Notations}\label{sec: well posedness}
Some prior estimates and the well-posed condition could be useful for the later analysis. We briefly review them in this section. We recall the transport equation with initial condition $g(x,v)$ and boundary condition $h(t,x,v)$:
\begin{align}\label{transport_equation}
\left\{ \begin{array}{ll}
\p_t u+v\cdot\nabla_x u+\textbf{E} \cdot\nabla_v u+q u  =S & \hbox{in } (0,T) \times \Omega\times \R^3\,,\\
u(0,x,v)  = g(x,v) & \hbox{in }   \Omega\times \R^3\,,\\
u = h  & \hbox{in }  (0,T) \times \Gamma_-\,.
\end{array}\right. 
\end{align}

It is clear from the equation \eqref{transport_equation} that the trajectories of particles are determined by the first three terms on the left hand side (LHS) of \eqref{transport_equation}. Along the trajectory, $q$ serves as a damping coefficient while $S$ is a source term. We define the two transport operators as follows:
\begin{equation}\label{def:transport_operator}
P_0 = \p_t +v\cdot\nabla_x+\textbf{E} \cdot\nabla_v\,,\quad P  = \p_t +v\cdot\nabla_x+\textbf{E} \cdot\nabla_v +q\, .
\end{equation}
 
Suppose at time $t$, a particle is placed at the initial position $x$ with the initial velocity $v$. We denote $X(s;t,x,v)$ and $V(s;t,x,v)$ the position and velocity of this particle at time $s$, with the initial condition set as
\[
(X(t;t,x,v), V(t;t,x,v)):=(x,v)\,.
\]
We have the trajectory determined by the Hamiltonian ODE system
\begin{align}\label{charateristic}
\dot{X}(s; t,x,v)=V(s;t,x,v)\,,\quad \dot{V}(s; t,x,v)=\textbf{E}(X(s;t,x,v))\,,
\end{align}
for $0 <s,\ t<\infty$, where $\dot{X}$ and $\dot{V}$ represent their derivatives with respect to time $s$. 

The particles, tracing backwards in time, either pick up information from boundary $h(t,x,v)$ or from the initial data $g(x,v)$. For the particles picking up information from the boundary, we define the backward exiting time $t_-$, position $x_-$ and velocity $v_-$ as follows:
\begin{definition}
For $(t,x,v)\in \R\times \overline{\Omega}\times\R^n$, we define the backward exit time $t_-(t,x,v)$ by
\begin{align}\label{backward exit time}
t_- (t,x,v) :=\sup\{s\geq 0:\ X(\tau;t,x,v)\in \Omega,\ \hbox{for all }\tau\in (t-s,t)\}\,,
\end{align}
and define the backward exit position $x_-$ and the backward exit velocity $v_-$ by
\[
x_-(t,x,v):= X(t-t_-(t,x,v);t,x,v)\,,\quad v_-(t,x,v):=V(t-t_-(t,x,v);t,x,v)\,.
\]
\end{definition}
 
\begin{figure}[ht]
	\centering
	\includegraphics[width=3.3in]{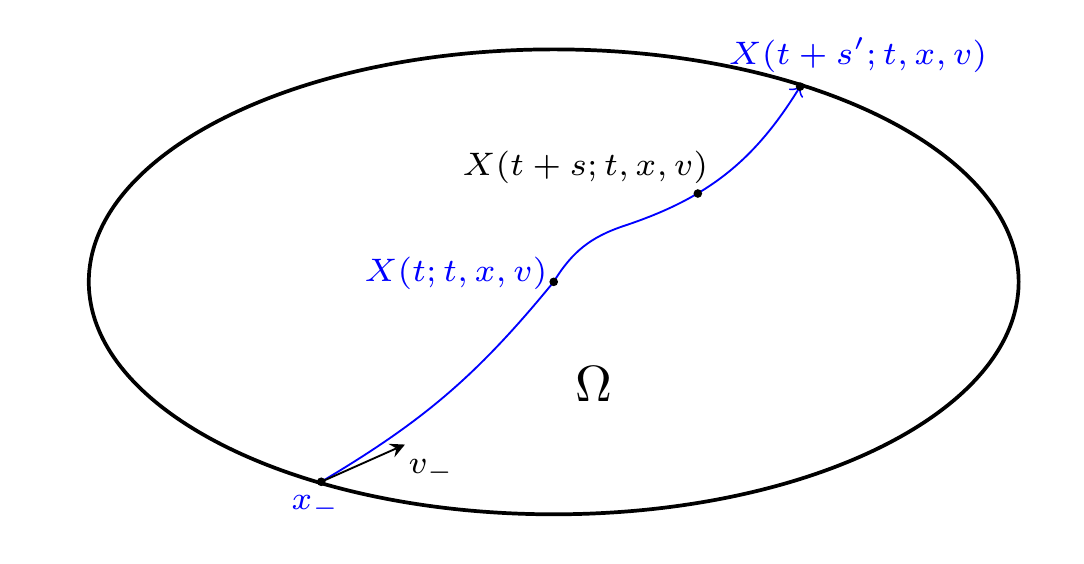} 
	\caption{\small The picture illustrates the particle travels along the characteristic in a bounded domain $\Omega$ and hits the boundary 
	at the backward exit point $x_-$ and velocity $v_-$ at time $t_-$. }
	\label{fig propagation}
\end{figure}

With these definitions, assuming the initial and boundary conditions are compatible, we have the following well-posed result for the problem \eqref{transport_equation}:
\begin{theorem}[\cite{CKL}, Proposition 2]\label{thm:wellposedness}
Let $\eta\in (0,1/4)$. Assume that the compatibility condition 
\begin{align}\label{compatibility condition}
g(x,v) = h(0, x,v)\ \ \hbox{ for }(x,v)\in\Gamma_- 
\end{align} and \eqref{En0} hold.
Suppose that $q\geq 0$ and
\begin{enumerate}
\item $\nabla_x g,\ \nabla_v g\in L^2(\Omega\times \R^3)$\,,
\item $\nabla_x S,\ \nabla_v S\in L^2([0,T]\times\Omega\times\R^3)$\,,
\item $e^{-\eta|v|^2}\nabla_x q,\ e^{-\eta|v|^2}\nabla_v q\in L^2([0,T]\times\Omega\times\R^n)$\,,
\item $e^{\eta|v|^2} g\in L^\infty(\Omega\times\R^3)$, $e^{\eta|v|^2} h\in L^\infty([0,T]\times\Gamma_-)$\,,
\item $e^{\eta|v|^2} S\in L^\infty([0,T]\times\Omega\times\R^3)$\,,
\item $\nabla_{x,v} t_-\p_t h,\ \nabla_{x,v}v_-\nabla_v h,\ \nabla_{x,v}x_-\nabla_x h,\ \nabla_{x,v}t_-(q h)\in L^2([0,T]\times \Gamma_-)$\,.
\end{enumerate}
Then for any $T>0$, there exists a unique solution $u$ to \eqref{transport_equation} such that 
\[
\nabla_{x,v} u\in C([0,T];L^2(\Omega\times\R^3))\cap L^1([0,T];L^2(\p\Omega))\,.
\]
\end{theorem}

We emphasize that all these assumptions are rather loose in the sense that the initial and boundary condition $g$ and $h$ can have roughly $e^{|v|^2}$ growth. However, the conclusion is strong in the sense that $\nabla_{x,v}u\in L^2(\Omega\times\R^3)$, despite having singularities, are still square integrable. In the remaining of this section, we provide certain discussions about the singular behavior for the completeness of the paper and refer the interested readers to \cite{Cao18, CKL} and the references therein for details. 

In fact, one can trace the dynamics of the particles around the boundary and make the singular behavior rather explicit. Since the trajectory is dominated by the operator $P_0$, we take the Vlasov equation as an example:
\begin{align}\label{equ Vlasov}
P_0u = 0\,.
\end{align}
Suppose the boundary condition is determined by some given function $h$ so that
\begin{align}\label{equ Vlasov bdry}
u(t,x,v)=h(t,x,v)\ \ \hbox{ for } (t, x,v)\in \R_+\times\Gamma_-\,,
\end{align}
then the explicit solution to the Vlasov equation \eqref{equ Vlasov} with \eqref{equ Vlasov bdry},
according to~\cite{Ukai86} is
\begin{align}\label{Vlasov solution}
u(t,x,v) =h(t-t_-(t,x,v), x_-(t,x,v),v_-(t,x,v))\ \ \ \hbox{for $t\geq t_-(t,x,v)$}\,.
\end{align}
Differentiating the equation \eqref{Vlasov solution} in $x$, we have the following lemma.
 
\begin{lemma}\label{lemma singular}
	Let $u$ be the solution to the Vlasov equation. For $(t,x,v)\in \R\times \overline\Omega\times\R^3$, if $n(x_-(t,x,v))\cdot v_-(t,x,v)$ is sufficiently small, then
	$$
	\nabla_x u(t,x,v)\sim {1\over n(x_-(t,x,v))\cdot v_-(t,x,v)}\,,
	$$
where $n(x_-(t,x,v))$ is the unit outer normal at the point $x_-(t,x,v)$ on $\p\Omega$.
\end{lemma} 
\begin{proof}
	For $(x_-(t,x,v), v_-(t,x,v))\in\Gamma_-$, the partial derivative of $x_-$ is parallel to the tangential direction, thus one has $\partial_{x_i} [x_-(t,x,v)]\cdot n(x_-(t,x,v))=0$. To reveal the singularities, we rely on the identities
	\begin{align}\label{partial x}
	\partial_{x_i}[ x_-(t,x,v)]
	&=\partial_{x_i}[X(t-t_-(t,x,v);t,x,v)] \notag\\
	&=-\partial_{x_i}t_-(t,x,v) \dot{X}(t-t_-(t,x,v);t,x,v)+ \partial_{x_i} X(t-t_-(t,x,v);t,x,v)
	\end{align}
	and 
	\begin{align}\label{partial v}
	\partial_{x_i}[ v_-(t,x,v)]
	&=\partial_{x_i}[V(t-t_-(t,x,v);t,x,v)]  \notag\\
	&=-\partial_{x_i}t_-(t,x,v) \dot{V}(t-t_-(t,x,v);t,x,v)+ \partial_{x_i} V(t-t_-(t,x,v);t,x,v)
	\end{align}
	for all $1\leq i\leq 3$. To see the singularity, we perform the inner product of \eqref{partial x} and the normal vector, then we have
	\begin{align*}
	0&=n(x_-(t,x,v))\cdot \partial_{x_i} [x_-(t,x,v)]\notag \\
	&=-\partial_{x_i} t_-(t,x,v) (v_-(t,x,v)\cdot n(x_-(t,x,v)) + \partial_{x_i} X(t-t_-(t,x,v);t,x,v)\cdot n(x_-(t,x,v)) \,,
	\end{align*}
	which leads to 
	\begin{align}\label{singular t}
	\partial_{x_i} t_-(t,x,v) = {\partial_{x_i} X(t-t_-(t,x,v);t,x,v)\cdot n(x_-(t,x,v)) \over  v_-(t,x,v)\cdot n(x_-(t,x,v))}\,.
	\end{align}
	Taking partial derivatives on \eqref{Vlasov solution}, we obtain
	\begin{align*}
	\partial_{x_i} u(t,x,v) = -\partial_{x_i} t_-(t,x,v)\p_t h + \p_{x_i}[x_-(t,x,v)]\partial_{x_i} h + \p_{x_i}[v_-(t,x,v)]\partial_{v_i} h\,.
	\end{align*}
	We then substitute \eqref{partial x}, \eqref{partial v}, and \eqref{singular t} into $\partial_{x_i} u$. Thus, it can be seen that
	the function $\nabla_x u$ has singularities when $n(x_-(t,x,v))\cdot v_-(t,x,v)=0$.
\end{proof}

Under the following assumptions on $\textbf{E}$, the singular behavior at the boundary can be avoided and then the $H^1$ estimate on the solution can be obtained, see for example \cite{Cao18, CKL}. 
\begin{lemma}[\cite{CKL}, Lemma 1]\label{lemma avoid singular}
Let $\Omega$ be a convex domain. Suppose that $\|\textbf{E}(x)\|_{C^1(\Omega)} <\infty$ and
\begin{align}\label{En0}
n(x) \cdot \textbf{E}(x) = 0 \quad\quad\text{for}\ \ x \in\partial\Omega\,,
\end{align}
where $n(x)$ is the unit outer normal vector at $x\in\p\Omega$.
Then for $x\in\partial\Omega$ and $n(x)\cdot v > 0$, we have
\[
n(x_-(t, x, v))\cdot v_-(t, x, v) < 0\,
\]
provided that $t + 1 \geq t_-(t, x, v)$.
\end{lemma}

Assume that $\textbf{E}$ satisfies the assumptions in Lemma~\ref{lemma avoid singular}, thanks to both Lemma~\ref{lemma singular} and Lemma~\ref{lemma avoid singular}, one can control the singularities of $\nabla_{x,v}u$ on the boundary, and this serves as one of the main components in showing the well-posedness in Theorem~\ref{thm:wellposedness}. We refer the interested readers to \cite{CKL} for the proof of Theorem~\ref{thm:wellposedness} and Lemma~\ref{lemma avoid singular}, and~\cite{Cao18} for the case when $\textbf{E}(x)\cdot n(x)>C_E>0$.

\section{Energy and Carleman Estimates}\label{sec:Carleman}
In this section we prepare the main ingredients for the proofs of Theorem \ref{main thm stability estimate} and \ref{thm: source}. In particular, we will derive the energy estimate and the Carleman estimate for equation~\eqref{transport_equation}. The application of these estimates will be explored in Section 4.

To a large extent, the energy estimate mainly follows from the integration by parts and some standard inequalities (Gr\"onwall and Cauchy-Schwarz). Moreover, we derive the Carleman estimate by designing a special weight function that enlarges/suppresses the information of the solution in time.


\subsection{Energy estimates}
We first state the Green's identity on the phase space:
\begin{lemma} [\cite{CKL}, Lemma 5]\label{green identity}
Suppose that $u\in L^2 ([0,T];L^2(\Omega\times\R^3))\cap L^2 ([0,T];L^2(\Gamma_-))$ and
\[
F(t,x,v):=\p_t u+v\cdot\nabla_x u+\textbf{E}\cdot\nabla_v u\in L^2([0,T];L^2(\Omega\times\R^3))\,.
\]
Then  
\[
u\in L^2 ([0,T];L^2(\Gamma_+))\cap C ([0,T];L^2(\Omega\times\R^3))\,.
\]
		Moreover, the following identity holds 
		\begin{align}\label{Green identity}
		&\int_{\Omega\times\R^3}|u(s,x,v)|^2 dxdv +\int^s_0\int_{\Gamma_+}|u|^2d\sigma dt \notag\\
		 =\ &\int_{\Omega\times\R^3} |u(0,x,v)|^2dxdv+\int^s_0\int_{\Gamma_-}|u|^2d\sigma dt+\int^s_0\int_{\Omega\times\R^3} F(t,x,v) u dxdvdt 
		\end{align}
		for almost every $s\in[0,T]$.
\end{lemma}
This Green's identity, combined with Gr\"onwall's inequality, allows us to obtain an energy estimate:
\begin{lemma}\label{lemma energy}
Suppose that $W \in L^2 ([0,T]\times\Omega\times\R^3)$, $q\in L^\infty (\Omega\times\R^3)$, and $h\in L^2 ([0,T];L^2(\Gamma_-))$. Let $u\in L^2 ([0,T]; L^2(\Omega\times\R^3))$ be the solution to the following problem
\begin{align}\label{transport_equation energy estimate}
\begin{cases}
\p_t u+v\cdot\nabla_x u+\textbf{E} \cdot\nabla_v u +q u =2W & \hbox{in } (0,T)\times \Omega\times \R^3\,,\\
u = h  &\hbox{in } (0,T)\times\Gamma_-\,.\\
\end{cases}
\end{align}
Then there exists a constant $C>0$, depending on $T$ and $\|q\|_{L^\infty}$, so that for every $0\leq s\leq T$, one has
\begin{align}\label{new Green identity}
&\int_{\Omega\times\R^3}|u(s,x,v)|^2 dxdv +\int^T_0\int_{\Gamma_+}|u|^2d\sigma dt\notag\\
 \leq\ & Ce^{CT} \LC \int_{\Omega\times\R^3} |u(0,x,v)|^2dxdv+\int^T_0\int_{\Gamma_-}|u|^2d\sigma dt+  \int^T_0\int_{\Omega\times\R^3}|W|^2 dxdvdt \RC.
\end{align}
\end{lemma}
\begin{proof}
We first apply \eqref{Green identity} for any $0\leq s\leq T$ and obtain
\begin{align}\label{CKL green}
     &\int_{\Omega\times\R^3}|u(s,x,v)|^2 dxdv +\int^s_0\int_{\Gamma_+}|u|^2d\sigma dt \notag \\
      =\  &\int_{\Omega\times\R^3} |u(0,x,v)|^2dxdv+\int^s_0\int_{\Gamma_-}|u|^2d\sigma dt \notag \\
      &+\int^s_0\int_{\Omega\times\R^3}\LC \p_t u+v\cdot\nabla_x u+\textbf{E}(x)\cdot\nabla_v u\RC u dxdvdt  \,.
\end{align}
We then multiply the equation \eqref{transport_equation energy estimate} by $u$ and integrate both sides so that 
\begin{align}\label{multiply u}
\int^s_0\int_{\Omega\times\R^3}\LC \p_t u+v\cdot\nabla_x u+\textbf{E} \cdot\nabla_v u\RC u +q|u|^2 dxdvdt = \int^s_0\int_{\Omega\times\R^3}2W udxdvdt\,.
\end{align}
Replacing the third term on the right hand side (RHS) of \eqref{CKL green} by identity \eqref{multiply u}, then one obtain
\begin{align}\label{Green identity 0}
&\int_{\Omega\times\R^3}|u(s,x,v)|^2 dxdv +\int^s_0\int_{\Gamma_+}|u|^2d\sigma dt \notag\\
 =\ &\int_{\Omega\times\R^3} |u(0,x,v)|^2dxdv+\int^s_0\int_{\Gamma_-}|u|^2d\sigma dt \notag\\
&  -\int^s_0\int_{\Omega\times\R^3}q|u|^2 dxdvdt + \int^s_0\int_{\Omega\times\R^3}2 W  u dxdvdt
\end{align}
for any $0\leq s\leq T$. We denote the energy $\mathcal{K}$ at time $s$ by 
$$
   \mathcal{K}(s):= \int_{\Omega\times\R^3}|u(s,x,v)|^2 dxdv ,\ \ 0\leq s\leq T\,.
$$
Applying the following inequality 
\begin{align*} 
   2 \int^s_0\int_{\Omega\times\R^3} W  udxdvdt\leq \int^s_0\int_{\Omega\times\R^3} |W|^2 dxdvdt+\int^s_0\int_{\Omega\times\R^3}|u|^2 dxdvdt\,,
\end{align*}
and the identity \eqref{Green identity 0}, it leads to
\begin{align} \label{E_0}
& \mathcal{K}(s) +\int^s_0\int_{\Gamma_+}|u|^2d\sigma dt\leq \alpha(s) +(1+\|q\|_{L^\infty})\int^s_0\mathcal{K}(t)dt\,,
\end{align}
where 
$$
    \alpha(s):=\mathcal{K}(0) +\int^s_0\int_{\Gamma_-}|u|^2d\sigma dt + \int^s_0\int_{\Omega\times\R^3}| W |^2  dxdvdt,\ \ 0\leq s \leq T\,.
$$ 
To apply Gr\"onwall's inequality, we temporarily drop the second term on the LHS of \eqref{E_0}. Since $\alpha$ is nondecreasing, we obtain
\[
   \mathcal{K}(s)\leq  \alpha(s)e^{T(1+\|q\|_{L^\infty})} \leq \alpha(T)e^{T(1+\|q\|_{L^\infty})}\,,\quad 0\leq s\leq T\,.
\]
Substituting it back to the RHS of \eqref{E_0} and using the fact again that $\alpha$ is nondecreasing, then we have
\[
\mathcal{K}(s)+\int^s_0\int_{\Gamma_+}|u|^2d\sigma dt \leq \alpha(T) + (1+\|q\|_{L^\infty})T\alpha(T)e^{T(1+\|q\|_{L^\infty})} ,\ \ 0\leq s\leq T\,,
\]
which completes the proof.	
\end{proof}

The following corollary follows immediately from the lemma above.
\begin{corollary}
Suppose $q\in L^\infty(\Omega\times\R^3)$ and $k(t,x,v)=k_0(x,v)k_1(t,x,v)$ with $k_0\in L^2(\Omega\times\R^3)$ and 
\begin{align}\label{k1 estimate1}
 \| k_1\|_{L^\infty([0,T]\times \Omega \times \R^3)}\leq C_3\,,
\end{align}
where $C_3$ is a positive constant. 
Let $u$ be the solution to the problem \eqref{transport_equation energy estimate} with $W$ replaced by $k$. Then the following estimate holds:
\begin{align}\label{equ energy}
      &\|u\|^2_{L^2([0,T]\times\Omega\times\R^3)} + \|u\|^2_{L^2([0,T]\times\Gamma_+)}  \notag\\
       \leq\ &  C \left( \|u(0,\cdot,\cdot)\|^2_{L^2(\Omega\times\R^3)}+ \|u\|^2_{L^2([0,T]\times\Gamma_-)} + \|k_0\|^2_{L^2(\Omega\times\R^3)}\right),
\end{align}
where the constant $C>$ depends on $T$, $\|q\|_{L^\infty}$, and $C_3$.
\end{corollary}
The proof is a direct application of Lemma \ref{lemma energy}, where one integrates both sides of \eqref{new Green identity} over the interval $[0,T]$. Another straightforward result is the following theorem which we omit the proof.

\begin{theorem}\label{thm control}
Suppose that
$\|k_1\|_{L^\infty([0,T]\times \Omega\times \R^3 )}\leq C_3$ and $k_0\in L^2(\Omega\times\R^3)$.
Let $u\in L^2([0,T]\times \Omega\times\R^3)$ be the solution to 
\begin{align}\label{transport_equation_1}
\left\{ \begin{array}{ll}
\p_t u+v\cdot\nabla_x u+\textbf{E}\cdot\nabla_v u+q u  =k_0(x,v) k_1(t,x,v) & \hbox{in }(0,T)\times \Omega\times \R^3 \,,\\
u(0,x,v)=0  &\hbox{in } \Omega\times\R^3 \,,\\
u = 0 &\hbox{in } (0,T)\times \Gamma_- \,.\\
\end{array}\right. 
\end{align}
Then 
\[
c\|u\|_{L^2([0,T]\times\Gamma_+)}  \leq  \|k_0\|_{L^2(\Omega\times\R^3)}\quad\hbox{and}\quad c\|u\|_{L^2([0,T]\times\Omega\times\R^3)}  \leq  \|k_0\|_{L^2(\Omega\times\R^3)}\,,
\]
where $c >0$ is a constant depending on $C_j,\ \Omega,\ \textbf{E}$, $q$, and $T$.
\end{theorem}
\begin{remark}
If we further assume that 
\[
\|u_2\|_{L^\infty([0,T]\times \Omega\times \mathcal{U})}\leq C_3\,,\quad \|S_0\|_{L^\infty([0,T]\times \Omega\times \mathcal{U})}\leq C_3
\]
in Theorem \ref{main thm stability estimate} and Theorem \ref{thm: source}, then from Theorem \ref{thm control}, the first estimate in \eqref{introduction estimate} and \eqref{diff source introduction estimate} can be replaced by 
\[
c\|u_1-u_2\|_{L^2([0,T]\times\Omega\times\R^3)}\leq  	\|q_1-q_2\|_{L^2(\Omega\times\R^3)}\,,
\]
and 
\[
c\|u_1-u_2\|_{L^2([0,T]\times\Omega\times\R^3)} \leq  	\|\widetilde S_1-\widetilde S_2\|_{L^2(\Omega\times\R^3)}\,,
\]
respectively.
\end{remark}


Carleman estimates typically rely on a good design of a ``weight function". Before presenting it, as a preparation, we here introduce the following related identity. 
The following lemma holds true for a certain function $\Psi$ related to the weight function, and its specific form will be designed in a later section. We note that the proof here is similar to that of Lemma 5 in~\cite{CKL}, but we do need adjustments to fit our setting.  
\begin{lemma}\label{lemma_Carleman}
Suppose $\Psi(x,v)$ 
satisfies 
\[
\|\Psi\|_{L^\infty(\overline{\Omega}\times \R^3)} \leq M_0\,,\quad \|v\cdot\nabla_x \Psi+\textbf{E}\cdot\nabla_v \Psi\|_{L^\infty(\overline{\Omega}\times \R^3)} \leq M_1\,
\]
for some positive constants $M_0$ and $M_1$. Denote
\[
F(t,x,v):=(\p_t w  +v\cdot\nabla_x w +\textbf{E}\cdot\nabla_vw )w\,.
\]
Under the assumption that
\[
w\in L^2 ([0,T];L^2(\Omega\times\R^3))\cap L^2 ([0,T];L^2(\Gamma_-))  \,,
\]
and that
\[
\p_t w  +v\cdot\nabla_x w +\textbf{E}\cdot\nabla_v w\in L^2 ([0,T]; L^2(\Omega\times\R^3))\,,
\]
we have
	\begin{align}\label{pre Carleman estimate}
	&2\int^T_0\int_{\Omega\times\R^3} \Psi  F  dxdv dt \notag\\
	 =\ &-\int^T_0\int_{\Omega\times\R^3} [v\cdot\nabla_x \Psi+\textbf{E}\cdot\nabla_v \Psi] |w|^2(t,x,v) dxdv dt+\int^T_0\int_{\Gamma_+} \Psi|w|^2 d\sigma dt \notag\\ & +\int_{\Omega\times\R^3} \Psi|w|^2(T,x,v) dxdv-\int_{\Omega\times\R^3} \Psi|w|^2(0,x,v) dxdv -\int^T_0\int_{\Gamma_-} \Psi|w|^2 d\sigma dt\,.
	\end{align}
\end{lemma}
\begin{proof}


We denote the function $\mathcal{H}$ by
$$
\mathcal{H}(t,x,v)= [v\cdot\nabla_x \Psi+\textbf{E}\cdot\nabla_v \Psi]|w|^2+2\Psi F  
$$
and observe that
\begin{align*}
&\mathcal{H}(t+s, X(t+s;t,x,v), V(t+s;t,x,v))\\ 
 =\ & {d\over ds}\Psi(X(t+s;t,x,v), V(t+s;t,x,v))|w(t+s, X(t+s;t,x,v), V(t+s;t,x,v))|^2\, .
\end{align*}
Thus, the function $\mathcal{H}$ is in $L^1([0,T]\times \Omega\times \R^3)$ which can be deduced from the hypothesis. Since $\mathcal{H}$ satisfies the condition in Lemma 4 in \cite{CKL} whose proof relies on the change of variables to the function $\mathcal{H}$, we immediately obtain the following identity
\begin{align}\label{energy_1}
&\int^T_0\int_{\Omega\times\R^3}\mathcal{H}(t,x,v)  dxdvdt \notag\\
=\ &  \int_{\Omega\times \R^3}\int^0_{-T\wedge t_-(T,x,v)} {d\over ds}  \Psi|w|^2 (T+s, X(T+s; T,x,v) , V(T+s;T,x,v)) dsdvdx\notag\\
&  +\int^T_0\int_{\Gamma_+}\int^0_{-t\wedge t_-(t,x,v)} {d\over ds}  \Psi|w|^2 (t+s, X(t+s; t,x,v) , V(t+s;t,x,v))  dsd\sigma dt \notag\\
=:\, &J_1+J_2\,,
\end{align}
where we used the notation $a\wedge b=\min\{a,b\}$. Let us first consider $J_1$. From a direct computation on $J_1$, we can derive
\begin{align*}
J_1&=\int_{\Omega\times \R^3}\int^0_{-T\wedge t_-(T,x,v)} {d\over ds}  \Psi|w|^2 (T+s, X(T+s; T,x,v) , V(T+s;T,x,v)) dsdvdx\notag\\
&=\int_{\Omega\times \R^3} \Psi(x,v) |w(T,x,v)|^2  dxdv - K_1-K_2\,,
\end{align*}
where
\begin{align*}
K_1&= \int_{\Omega\times\R^3} 1_{T\geq t_-(T,x,v)} \Psi(x_-,v_-)|w(T-t_-,x_-,v_-)|^2 dxdv \,,\\
K_2&= \int_{\Omega\times\R^3} 1_{T< t_-(T,x,v)} \Psi(X(0; T,x,v) , V(0;T,x,v))|w(0, X(0; T,x,v) , V(0;T,x,v))|^2 dxdv\,.
\end{align*}
In addition, we can also obtain an identity of $J_2$ as follows:	 
\begin{align*}
J_2&=\int^T_0\int_{\Gamma_+}\int^0_{-t\wedge t_-(t,x,v)} {d\over ds}  \Psi|w|^2 (t+s, X(t+s; t,x,v) , V(t+s;t,x,v))  dsd\sigma dt \notag\\
&=\int^T_0\int_{\Gamma_+} \Psi( x,v)|w(t,x,v)|^2  d\sigma dt -K_3-K_4\,,
\end{align*}
where
\begin{align*}
K_3 &= \int^T_0\int_{\Gamma_+} 1_{t\geq t_-(t,x,v)} \Psi(x_-,v_-)|w(t-t_-,x_-,v_-)|^2 d\sigma dt\,,\\
K_4 &= \int^T_0\int_{\Gamma_+} 1_{t< t_-(t,x,v)}\Psi(X(0; t,x,v) , V(0;t,x,v))|w(0, X(0; t,x,v) , V(0;t,x,v))|^2 d\sigma dt \,.
\end{align*}
It is indicated in the proof of Lemma 5 in \cite{CKL} that
$$
    K_2+K_4 =  \int_{\Omega\times\R^3}\Psi(x,v) |w(0,x,v)|^2 dxdv 
$$
and 
$$
    K_1+K_3 = \int^T_0 \int_{\Gamma_-}\Psi(x,v)|w(t,x,v)|^2 d\sigma dt\,.
$$
Therefore, the proof is complete by putting $J_1+J_2$ back to \eqref{energy_1}.
\end{proof}

\subsection{Carleman estimates}
The key to deriving the Carleman estimates is to find a suitable weight function, and we discuss it in this section. We also refer to \cite{Yamamotobook} for the application of Carleman estimates to inverse problems in different settings.

We first choose a weight function $\varphi\in C^2([0,T]\times \overline{\Omega}\times \R^3)$ of the following form
\begin{align}\label{phase}
    \varphi(t,x,v)=-\beta t+\varphi_0(x,v)\,,\quad\text{with}\quad \beta>0\quad\text{and}\quad\varphi_0\in C^2(\overline{\Omega}\times \R^3)\,,
\end{align} 
and then we define the function $\Psi$ by acting the transport operator $P_0$ (defined in \eqref{def:transport_operator}) on $\varphi$, namely,
\begin{align}\label{def Psi} 
   \Psi :=P_0\varphi = \p_t\varphi +v\cdot\nabla_x \varphi + \textbf{E}\cdot \nabla_v\varphi= -\beta+v\cdot\nabla_x\varphi_0 +\textbf{E}\cdot\nabla_v\varphi_0\,.
\end{align}

We now impose some assumptions on $\varphi$ and $\Psi$, and they are needed for the stability estimate, to be presented  later in Section~\ref{proof of thm}, to be carried through.
\begin{hypothesis}\label{hypo}
For an open subset $\mathcal{V}\subset\R^3$, there exists a function $b(x)$ such that $\textbf{E}(x)=-\nabla_x b(x)$ and the following statements hold:
\begin{enumerate}
\item One has $R>r>0$ with
$$
R=\sup_{\overline{\Omega}\times \mathcal{V}}  \varphi_0(x,v) >0,\ \ r=\inf_{\overline{\Omega}\times \mathcal{V}} \varphi_0(x,v) >0\,.
$$ 
\item Let $T>(R-r)/\beta$. If $\varepsilon>0$ is sufficiently small, then there exist constants $\alpha_0,\alpha_1$ such that $ 0<\alpha_0<\alpha_1<r$,
\begin{align}\label{varphi T}
\sup_{\overline{\Omega}\times \mathcal{V}} \varphi(t,x,v) \leq \alpha_0  \ \ \hbox{for }T-2\varepsilon \leq t\leq T  
\end{align} 
and
$$
\sup_{\overline{\Omega}\times \mathcal{V}} \varphi(t,x,v) \geq \alpha_1  \ \ \hbox{for }0 \leq t\leq \varepsilon\,. 
$$
\item For some constant $\gamma_0>0$, the function $\Psi$ satisfies
\begin{align*}
\Psi(x,v)  \geq\gamma_0 >0\ \ \hbox{for any }(x,v)\in \overline{\Omega} \times \mathcal{V}\,.
\end{align*}
\item Moreover, the function $\Psi$ satisfies 
\begin{align*}
\sup_{\overline{\Omega}\times \mathcal{V}}|\Psi(x,v)|\leq M_0
\end{align*}
and 
$$
\sup_{\overline{\Omega}\times \mathcal{V}}|(v\cdot\nabla_x +\textbf{E} \cdot\nabla_v)\Psi(x,v)| \leq M_1 
$$
for some positive constants $M_0$ and $M_1$.
\end{enumerate} 
\end{hypothesis} 

We argue there exist pairs of $(\varphi,\Psi)$ that satisfy the hypothesis. In fact, to make $\varphi$ satisfying Hypothesis \ref{hypo}, it is crucial to choose $\varphi_0$ properly so that
\begin{align*}
\inf_{\overline\Omega\times \mathcal{V}} (v\cdot \nabla_{x} \varphi_0 +\textbf{E}\cdot\nabla_v \varphi_0 ) = \mu>0 \,.
\end{align*}
In addition, by choosing $\beta$ with $0<\beta<\mu$, it implies that
\begin{align}\label{symbol Psi}
\Psi(x,v)=-\beta+v\cdot \nabla_{x} \varphi_0 +\textbf{E}\cdot\nabla_v \varphi_0\geq -\beta +\mu >0 \, \quad\hbox{for all } (x,v)\in \overline\Omega\times \mathcal{V}\,.
\end{align}
 
In the following lemma, we design one particular example of $\varphi$ and $\Psi$ so that they fulfill all conditions in this hypothesis. This is simply to demonstrate that the set is not empty. There are other possible examples, but we do not discuss them in the paper.

\begin{lemma}\label{conditions}
 Suppose that $\| \textbf{E}\|_{C^1(\overline{\Omega})}\leq m$ for some constant $m>0$. Suppose that $\diam(\Omega)\leq \delta$ and for any $x=(x_1,x_2,x_3)\in \Omega$, $x$ satisfies $x_1>d>0$. Let $0<a<b<\infty$ with $\beta+2\delta m< a$. We choose the set
 \begin{align}\label{set V}
 \mathcal{V}=\{v=(v_1,v_2,v_3)\in\R^3:\ a\leq  v_1^2\leq b,\ v_1>0\}\,.
 \end{align} 
 Then the function
\begin{align*}
    \varphi(t,x,v)=-\beta t+ x_1v_1 
\end{align*} 
satisfies Hypothesis \ref{hypo}.
\end{lemma}
\begin{proof}
It is clear that
\[
R=\sup_{\overline{\Omega}\times \mathcal{V}} (x_1v_1)> 0\ \ \hbox{and}\ \ r=\inf_{\overline{\Omega}\times \mathcal{V}}(x_1v_1)>0\,.
\]
One chooses the observed time $T$ satisfying
\[
T>{R-r\over \beta}\,.
\]
Therefore, it leads to for any $(x,v),\ (x',v')\in \overline{\Omega}\times\mathcal{V}$:
\begin{align*}
&\varphi(T,x,v)=-\beta T+x_1v_1\leq -\beta T+R <r \leq \varphi(0,x',v')\,.
\end{align*}
Since $\varphi$ is continuous in $t$, there exists constants $\varepsilon>0$, $\alpha_0$, and $\alpha_1$ such that
\[
0<\alpha_0<\alpha_1<r\,,
\]
then one has
\begin{align*}
\sup_{\overline{\Omega}\times\mathcal{V}}\varphi(t,\cdot,\cdot) \leq \alpha_0, \ \ \hbox{for }T-2\varepsilon \leq t\leq T \,, 
\end{align*}
and
\[
\sup_{\overline{\Omega}\times\mathcal{V}} \varphi(t,\cdot,\cdot) \geq \alpha_1, \ \ \hbox{for }0 \leq t\leq \varepsilon\,.
\]

By the definition of $\Psi$, we have
\[
\Psi(x,v) =-\beta + v_1^2 +\textbf{E}\cdot(x_1,0,0) \geq 2\delta m  -|x|\|\textbf{E}\|_{L^\infty(\Omega)}\geq \delta m\ \ \hbox{for } (x,v)\in \overline{\Omega}\times \mathcal{V}\,,
\]
and $|\Psi|\leq M_0$ for some constant $M_0>0$. In addition, we denote $\textbf{E}=(E_1,E_2,E_3)$ and then we obtain
\[
v\cdot\nabla_x \Psi+\textbf{E}\cdot \nabla_v \Psi = v\cdot \nabla_x (E_1 x_1) +2v_1 E_1\leq m\sqrt{b}+m\delta \sqrt{b}+2bm<\infty\,.
\]
Therefore all conditions in Hypothesis \ref{hypo} hold true.
\end{proof}

Recall the definition of the transport operator $P_0$ in \eqref{def:transport_operator}. 
We define the function $w$ by
\[
w(t,x,v):=e^{s\varphi(t,x,v)} u(t,x,v) 
\]
and the operator $L$ by
\begin{equation}\label{eqn:def_L}
L(\cdot):= e^{s\varphi(t,x,v)}P_0(e^{-s\varphi(t,x,v)}\cdot) 
\end{equation}
for $s>0$. Then it is clear that
\begin{align}\label{operator L}
Lw=e^{s\varphi(t,x,v)}P_0u\,.
\end{align}

We are now ready to present the Carleman estimates.
\begin{lemma}\label{Carleman estimate} 
Suppose that $q=q(x,v)\in L^\infty(\Omega\times\R^3)$ satisfies $\|q\|_{L^\infty}\leq C_0$ for some positive constant $C_0$. Suppose that the hypotheses in Lemma \ref{lemma_Carleman} and the Hypothesis \ref{hypo} with $\mathcal{V}=\R^n$ hold for the functions $w$, $\varphi$ and $\Psi$. Then for sufficiently large $s>0$, we have 
\begin{align}\label{Carleman 1}
& s\int_{\Omega\times\R^3} \Psi  |u|^2(0,x,v)  e^{2s\varphi(0,x,v)} dxdv+s^2\int^T_0\int_{\Omega\times\R^3}  |\Psi|^2 |u|^2 e^{2s\varphi} dxdvdt \notag\\
\leq\ &  s\int^T_0\int_{\Omega\times\R^3}|v\cdot\nabla_x \Psi+\textbf{E}\cdot\nabla_v \Psi|   |u|^2 e^{2s\varphi} dxdvdt   \notag\\
\quad&+ s\int_{\Omega\times\R^3} \Psi|u|^2(T,x,v) e^{2s\varphi}dxdv+s\int^T_0\int_{\Gamma_+}\Psi|u|^2 e^{2s\varphi} d\sigma dt \notag\\
\quad&-s\int^T_0\int_{\Gamma_-} \Psi|u|^2 e^{2s\varphi}d\sigma dt +\int^T_0\int_{\Omega\times\R^3} |P_0 u|^2 e^{2s\varphi}  dxdvdt\,.
\end{align}
Moreover, when $s$ is sufficiently large, for some constant $c_0>0$ independent of $s$, we have
\begin{align}\label{Carleman 2}
&s\int_{\Omega\times\R^3} \Psi |u|^2(0,x,v)  e^{2s\varphi(0,x,v)} dxdv+c_0 s \int^T_0\int_{\Omega\times \R^3} |u|^2 e^{2s\varphi} dxdvdt \notag\\
\leq\ & s\int_{\Omega\times\R^3} \Psi|u|^2(T,x,v) e^{2s\varphi}dxdv+ s\int^T_0\int_{\Gamma_+} \Psi|u|^2 e^{2s\varphi} d\sigma dt \notag\\
\quad&-s\int^T_0\int_{\Gamma_-} \Psi|u|^2  e^{2s\varphi}d\sigma dt +2\int^T_0\int_{\Omega\times\R^3} |P u|^2 e^{2s\varphi} dxdv dt\,.
\end{align}
\end{lemma}
\begin{proof}
We first note that, according to the definition of $L$ in~\eqref{eqn:def_L}, the following equation holds:
\begin{align*}
Lw &= e^{s\varphi(t,x,v)} P_0 ( e^{-s\varphi(t,x,v)}w )\\
&=P_0w -s  (\p_t \varphi +v\cdot\nabla_x\varphi +\textbf{E}\cdot\nabla_v\varphi)  w\\
&=P_0w -s\Psi w \,.
\end{align*}
We perform the integration by parts and then obtain the following estimate:
\begin{align}\label{proof Carleman 1}
&\int^T_0\int_{\Omega\times\R^3}  |P_0 u|^2 e^{2s\varphi} dxdvdt = 	\int^T_0\int_{\Omega\times\R^3} |Lw|^2 dxdvdt \notag\\
=\, &  \int^T_0\int_{\Omega\times\R^3} |\p_t w +v\cdot\nabla_xw +\textbf{E} \cdot\nabla_vw|^2 dxdvdt +s^2\int^T_0\int_{\Omega\times\R^3}  |\Psi|^2 w^2 dxdvdt\notag\\
\quad& -2s\int^T_0\int_{\Omega\times\R^3}  \Psi w (\p_t w +v\cdot\nabla_xw +\textbf{E} \cdot\nabla_vw) dxdvdt\notag\\
\geq\, & -2s\int^T_0\int_{\Omega\times\R^3} \Psi w (\p_t w +v\cdot\nabla_xw +\textbf{E} \cdot\nabla_vw) dxdvdt \notag\\
\quad&+s^2\int^T_0\int_{\Omega\times\R^3}  |\Psi|^2 w^2 dxdvdt\,.
\end{align} 
Moreover, applying Lemma \ref{lemma_Carleman}, we rewrite the first term on the RHS of \eqref{proof Carleman 1} as
\begin{align}\label{proof Carleman 2}
&-2s\int^T_0\int_{\Omega\times\R^3} \Psi w (\p_t w +v\cdot\nabla_xw +\textbf{E}\cdot\nabla_vw) dxdvdt \notag\\
=\ &  s\int^T_0\int_{\Omega\times\R^3}  (v\cdot\nabla_x \Psi+\textbf{E}\cdot\nabla_v \Psi ) |w|^2  dxdvdt - s\int^T_0\int_{\Gamma_+} \Psi|w|^2 d\sigma dt\notag\\ 
&-s\int_{\Omega\times\R^3} \Psi|w|^2(T,x,v) dxdv+s\int_{\Omega\times\R^3} \Psi|w|^2(0,x,v) dxdv +s\int^T_0\int_{\Gamma_-}\Psi|w|^2 d\sigma dt\,.
\end{align}
Substituting \eqref{proof Carleman 2} and $w =e^{s\varphi}u$ into \eqref{proof Carleman 1}, we obtain
\begin{align*}
&s \int_{\Omega\times\R^3} \Psi |u|^2(0,x,v)e^{2s\varphi(0,x,v)}dxdv+s^2 \int^T_0\int_{\Omega\times\R^3} |\Psi|^2 |u|^2e^{2s\varphi} dxdvdt\\
\leq\ & s\int^T_0\int_{\Omega\times\R^3}|v\cdot\nabla_x \Psi+\textbf{E}\cdot\nabla_v \Psi|  |u|^2 e^{2s\varphi} dxdvdt +s\int_{\Omega\times\R^3} \Psi|u|^2(T,x,v) e^{2s\varphi}dxdv\\
\quad&+s\int^T_0\int_{\Gamma_+} \Psi|u|^2 e^{2s\varphi} d\sigma dt-s\int^T_0\int_{\Gamma_-} \Psi|u|^2  e^{2s\varphi}d\sigma dt +\int^T_0\int_{\Omega\times\R^3} |P_0 u|^2 e^{2s\varphi}  dxdvdt\,,
\end{align*}
which gives~\eqref{Carleman 1}.

To obtain the second estimate \eqref{Carleman 2}, we first replace $P_0u$ by
\[
|P_0u|^2\leq 2|P u|^2 + 2|qu|^2
\]
in the RHS of \eqref{Carleman 1}, where $P$ is defined in \eqref{def:transport_operator}. From the Hypothesis \ref{hypo}, $\Psi$ satisfies
\[
| v\cdot\nabla_x \Psi+\textbf{E} \cdot\nabla_v \Psi|\leq M_1\,\quad\text{in }\;\overline{\Omega}\times \R^3 
\]
and $ \Psi \geq \gamma_0 >0$ in $\overline{\Omega}\times \R^3$. Thus, for a large $s$, we can absorb the following two terms
\[
s\int^T_0\int_{\Omega\times\R^3} 	| v\cdot\nabla_x \Psi+\textbf{E} \cdot\nabla_v \Psi | |u|^2e^{2s\varphi} dxdvdt\,\quad\text{and}\quad \int^T_0\int_{\Omega\times\R^3} |qu|^2 e^{2s\varphi} dxdvdt
\]
in the RHS of \eqref{Carleman 1} into the LHS of \eqref{Carleman 1}. Thus, we have 
\begin{align*}
& \int^T_0\int_{\Omega\times \R^3} \LC  s^2|\Psi|^2- sM_1 - 2|q|^2	 \RC |u|^2 e^{2s\varphi} dxdvdt\\
\geq \ & (s^2\gamma_0^2 -sM_1 -2\|q\|^2_{L^\infty}) \int^T_0\int_{\Omega\times \R^3} |u|^2 e^{2s\varphi}dxdvdt\\
\geq \ &c_0s \int^T_0\int_{\Omega\times \R^3}  |u|^2 e^{2s\varphi}dxdvdt 
\end{align*} 
for some constant $c_0>0$, independent of $s$, provided that $s$ is sufficiently large. This completes the proof of~\eqref{Carleman 2}.
\end{proof}
From the proof above, we see that the lower order term does not affect the Carleman estimates if the coefficient $q$ is bounded, and that the inequality \eqref{Carleman 2} holds valid uniformly for any sufficiently large $s>0$.

\section{Reconstruction of Parameters}\label{proof of thm}
The energy estimates and the Carleman estimates from the previous section enable us to demonstrate stability of the reconstruction of the source and the absorption coefficient. Before presenting the proofs of the main theorems in Section \ref{proofs main thms}, assuming $\varphi$ and $\Psi$ satisfy Hypothesis~\ref{hypo} for an open set $\mathcal{V}$ in the velocity field, in Section~\ref{key lemmas}, we first provide two key lemmas which give the control of the parameter $k_0$ by utilizing the boundary measurement only.

\subsection{Key lemmas}\label{key lemmas}
We start with a special case where the force field $\textbf{E}$ has the form: $\textbf{E}(x)=(0,E_2,E_3)$. This is a pseudo 3D case in which there is no acceleration in the $v_1$ direction. It is a standard practice when the plasma particles are confined in a 3D system with symmetry in 1D~\cite{2Dplasma,semi_lag_vlasov}.


%
\begin{lemma}\label{main result 1}
Let $q=q(x,v)\in L^\infty(\Omega\times\R^3)$ satisfy $\|q\|_{L^\infty}\leq C_0$ for some positive constant $C_0$. Suppose that
\[
k(t,x,v):=k_0(x,v)k_1(t,x,v)\,,
\]
with $k_0\in L^2(\Omega\times\R^3)$ satisfying
\[
   \supp k_0(x,\cdot)\subset \mathcal{U}\ \ \hbox{for any }x\in \Omega \,,
\]
and $k_1$ satisfying
\begin{align}\label{k1 estimate}
\| k_1(0,x,v)\|_{L^\infty( \Omega \times \mathcal{U})}\leq C_1,\ \ \|\p_t k_1\|_{L^\infty([0,T]\times \Omega\times \mathcal{U})}\leq C_1,\ \ \hbox{and }\ \inf_{ \Omega \times \mathcal{U}}|k_1(0,x,v)|\geq C_2
\end{align}
for some constants $C_1,\ C_2>0$. 

Furthermore, assuming $u\in L^2([0,T]\times\Omega\times\R^3)$ is the solution to the pseudo 3D transport equation
\begin{equation}\label{boltzmann2}
\begin{cases}
\p_t u+v\cdot\nabla_x u+\textbf{E}\cdot\nabla_v u+q u=k &\hbox{in }(0,T)\times \Omega\times \R^3\,,\\
u(0,x,v)= 0  &\hbox{in }  \Omega\times \R^3 \,, \\
u=h&\hbox{in }\in\R_+\times\Gamma_- \,,
\end{cases}
\end{equation}
with $h,\ \p_t u\in L^2 ([0,T];L^2(\Gamma_-))$. If
\begin{align}\label{bound partial u}
\|\p_tu\|_{L^2([0,T]\times\Omega\times\R^3)}\leq \mathcal{M}<\infty 
\end{align}
for some fixed constant $\mathcal{M}>0$, then there exists an upper bound of $k_0$ in the $L^2$ norm:
\begin{align}\label{stability k0} 
\int_{\Omega\times\R^3}   | k_0|^2  dxdv
\leq C\mathcal{M}^{2-2\theta}\left(\int^T_0\int_{\Gamma_+\cup \Gamma_-}  |\p_t u|^2 d\sigma dt\right)^{\theta}
\end{align}
for some $\theta\in(0,1)$ and some constant $C>0$.
\end{lemma}
\begin{proof} 
In the pseudo 3D case, $\textbf{E}=(0,E_2,E_3)$. We denote $v=(v_1,v_2,v_3)$ and then choose a smooth cut-off function $\psi(v)\equiv\psi(v_1)$ in $v_1\in \R^3$ direction satisfying $\psi=1$ in $\mathcal{U}$, $\supp(\psi)\subset\mathcal{V}$ and that $|\psi|\leq 1$ in $\R^3$. Thus, it is clear that $\nabla_v\psi$ is orthogonal to the force $\textbf{E}$, namely,
$$
\textbf{E}\cdot \nabla_v \psi=0\,.
$$
From the help of the compact support of $\psi$ on the direction $v_1$, one can find a suitable weight function $\varphi$, for example, the one in Lemma \ref{conditions}, such that $\varphi$ satisfies the Hypothesis \ref{hypo}.

We choose, moreover, smooth cut-off function $\chi$ in time $t$ so that 
\begin{align*}
\chi(t) =\left\{\begin{array}{cc}1, & 0\leq t\leq T-2\varepsilon  \,,\\
0, & T-\varepsilon\leq t\leq T\,.\\
\end{array} \right. 
\end{align*}
We consider the function
\[
\tilde u=\chi(t) \psi(v)  \p_t u\,.
\]
Then $\tilde u$ satisfies the equation
\begin{align*}
P(\tilde u) &= \chi \psi P(\p_t u)  +  \psi  \p_t\chi  \p_t u + \chi (\textbf{E}\cdot \nabla_v \psi)  \p_t u \\
&= \chi \psi \p_t k+  \psi  \p_t \chi\p_t u\,,
\end{align*}
where we used the fact that $\textbf{E}\cdot \nabla_v\psi=0$.

Due to the cut-off function $\chi$ in time, one has $\tilde u(T,x,v)=0$. Furthermore, since $u(0,x,v)=0$, it implies from the equation \eqref{boltzmann2} that  $\p_tu(0,x,v)=k_0(x,v)k_1(0,x,v)$ which leads to the initial data of $\tilde u$, that is,
\[
	 \tilde u(0,x,v)= \psi \p_t u(0,x,v)= \psi k_0(x,v)k_1(0,x,v)\,.
\]

Applying the estimate \eqref{Carleman 2} to the function $\tilde{u}$ and the lower bound (3) of $\Psi$ in Hypothesis \ref{hypo}, we obtain 
\begin{align}\label{partial u estimate}
&s\gamma_0\int_{\Omega\times\R^3} |\psi  \p_t u(0,x,v)|^2 e^{2s\varphi(0,x,v)} dxdv  \notag\\ 
 \leq\ &  Cs\int^T_0\int_{\Gamma_+}\Psi|\chi\psi \p_t u|^2 e^{2s\varphi} d\sigma dt -Cs\int^T_0\int_{\Gamma_-} \Psi|\chi\psi \p_t u|^2 e^{2s\varphi}d\sigma dt \notag\\
&  +C\int^T_0\int_{\Omega\times\R^3} ( |\chi \psi  \p_t k|^2+ | \psi \p_t \chi\p_t u  |^2 ) e^{2s\varphi}  dxdvdt 
\end{align}
for some constant $C>0$ independent of $s$.	
We will give an estimate for the RHS of \eqref{partial u estimate}. We first denote
$$
\mathcal{N}^2:=  \int^T_0\LC \int_{\Gamma_+} +\int_{\Gamma_-}\RC|\psi\p_t u|^2d\sigma dt \,,
$$
then the first and the second terms can be bounded by
\begin{align}\label{boundary estimate}
	  Cs\int^T_0\LC \int_{\Gamma_+} -\int_{\Gamma_-}\RC\Psi|\chi \psi \p_t u|^2 e^{2s\varphi} d\sigma dt \leq Ce^{2s m} \mathcal{N}^2\,, 
\end{align} 
	where we applied $\|\varphi \|_{L^\infty(\Gamma_\pm)}\leq m$ in the support of $\psi$ and $\|\Psi \|_{L^\infty(\overline{\Omega}\times\mathcal{V})}\leq M_0$, according to Hypothesis bound (4), provided that $s$ is sufficiently large.
    
To estimate the fourth term on the right of \eqref{partial u estimate}, we use the upper bound \eqref{varphi T} for $\varphi$ and the fact that $\p_t \chi=0$ on $0\leq t\leq T-2\varepsilon$ and $T-\varepsilon\leq t\leq T$. Thus, we obtain
\begin{align}\label{u estimate}
\int^T_0\int_{\Omega\times\R^3}  | \psi \p_t \chi \p_t u|^2 e^{2s\varphi} dxdvdt
&\leq C\int^{T-\varepsilon}_{T-2\varepsilon}\int_{\Omega\times\R^3}  | \psi \p_t u|^2 e^{2s \alpha_0} dxdvdt \notag\\
&\leq Ce^{2s\alpha_0}  \|\psi\p_t u\|_{L^2([0,T]\times\Omega\times\R^3)}^2\,  
\end{align}
for sufficiently large $s>0$. 

For the third term in the RHS of \eqref{partial u estimate}, we have the estimate
\begin{align}\label{k estimate}
    \int^T_0\int_{\Omega\times\R^3}  |\chi \psi\p_t k|^2 e^{2s\varphi}dxdvdt 
    &\leq \int^T_0 \int_{\Omega\times\R^3} |\psi k_0(x,v)|^2|\p_t k_1(t,x,v)|^2 e^{2s\varphi} dxdvdt\notag\\
    &\leq C_1^2 T \int_{\Omega\times\R^3} |\psi k_0(x,v)|^2 e^{2s\varphi(0,x,v)} dxdv\,,
\end{align}
where we use the fact that $\varphi(t,x,v)\leq \varphi(0,x,v)$ for all $0\leq t \leq T$ and also \eqref{k1 estimate} in the last inequality.
From \eqref{partial u estimate}-\eqref{k estimate}, and \eqref{k1 estimate}, we can derive
   \begin{align*}
   &s\int_{\Omega\times\R^3}   |\psi k_0|^2 e^{2s\varphi(0,x,v)} dxdv\notag\\
   \leq\ &  CT\int_{\Omega\times\R^3}  |\psi k_0|^2 e^{2s\varphi(0,x,v)}dxdv + Ce^{2s\alpha_0} \|\psi \p_t u\|_{L^2([0,T]\times\Omega\times\R^3)}^2   + Ce^{2s m} \mathcal{N}^2 \,.
    \end{align*} 
This implies that 
\begin{align*}
 (s-CT)\int_{\Omega\times\R^3}  |\psi k_0|^2 e^{2s\varphi(0,x,v)} dxdv\leq  Ce^{2s\alpha_0} \|\psi\p_t u\|_{L^2([0,T]\times\Omega\times\R^3)}^2  + Ce^{2s m} \mathcal{N}^2 \, 
\end{align*} 
for large $s>CT$. Moreover, since $\varphi(0,x,v)\geq\alpha_1 $ and $\psi=1$ in $\mathcal{U}$, we can further get 
\begin{align}\label{k0 estimate}
\int_{\Omega\times\R^3}   |k_0|^2  dxdv &\leq  Ce^{2s \alpha_0-2s\alpha_1}\|\p_t u\|_{L^2([0,T]\times\Omega\times\R^3)}^2 + Ce^{2sm-2s\alpha_1} \mathcal{N}^2   \notag\\
&=  Ce^{-s \alpha^*}  \mathcal{M}^2  + Ce^{ s\beta^*} \mathcal{N}^2 \,,
   \end{align} 
where $\alpha^*:=2 \alpha_1-2 \alpha_0 >0$ and $\beta^*:=2 m-2 \alpha_1>0$.
	
We consider the following two cases:
\begin{enumerate}
\item $ \mathcal{M} \leq \mathcal{N}$\,,
\item  $\mathcal{M}> \mathcal{N}$\,.
\end{enumerate}
For case (1), we can derive from \eqref{k0 estimate} that
\begin{align}\label{k0 estimate1}
 \int_{\Omega\times\R^3}   |k_0|^2  dxdv \leq C (e^{-s \alpha^*} + e^{s\beta^*} ) \mathcal{N}^2 \,. 
\end{align} 
As for case (2), we balance two terms by letting $e^{-s \alpha^*} \mathcal{M}^2=  e^{s\beta^*} \mathcal{N}^2$, then we have
\[
    s=  {2\ln { \mathcal{M} \over \mathcal{N}  } \over \alpha^* + \beta^*}\,,
\]
which leads \eqref{k0 estimate} to 
\[
     \int_{\Omega\times\R^3}   | k_0|^2  dxdv \leq 2C  \mathcal{M}^{2-2\theta} \mathcal{N}^{2\theta}\,,
\] 
where $\theta= { \alpha^*\over \alpha^*+\beta^*}\in (0,1)$. This completes the proof of Lemma \ref{main result 1}.
\end{proof}
 
\begin{remark}\label{sec 4 remark}
Several comments are in line:
\begin{itemize}
\item Note that in Lemma \ref{main result 1}, to obtain the Carleman estimate, we used the technique borrowed from~\cite{Yamamoto2016}, and introduced the cut-off functions. In \cite{Yamamoto2016}, only cut-off function $\chi$ is utilized since the integrals of the Carleman estimate are over a bounded domain. However, in our case the integrals are over the whole space $\R^3$ for $v$ in \eqref{Carleman 2}, thus we introduce two cut-off functions $\chi$ in time and $\psi$ in velocity to control the integral in the velocity field. This is motivated by the average lemma.    
%
%
\item In the proof of Lemma~\ref{main result 1}, we do require a pseudo 3D case where $\text{E}$ has one component that is trivial. Similar argument can be applied if the electric field is either $\textbf{E}(x)=(E_1,0,E_3)$ or $\textbf{E}(x)=(E_1,E_2,0)$ with the corresponding adjusted domain $\Omega$ in $\{x\in \R^3:\ x_j>d\}$ for $j=2,3$, but one also needs to adjust the cut-off function $\psi=\psi(v_j)$ so that $\textbf{E}\cdot \nabla_v \psi=0$. Following the proof of Lemma \ref{main result 1}, we conclude with the same stability estimate for $k_0$ in \eqref{stability k0}.
\item We would like to note that in the most general case, one considers the field $\textbf{E}$ does not have a trivial component.  
As long as for such $\textbf{E}$, there exists a weight function $\varphi$ satisfying Hypothesis \ref{hypo} for the case $\mathcal{V}=\R^n$, then one can still derive the same estimate \eqref{stability k0} as well as \eqref{k0 stability estimate} without introducing the cut-off function $\psi$. 
\end{itemize}
\end{remark}

In Lemma~\ref{main result 1}, we assume the boundedness of $\p_t u$ in~\eqref{bound partial u}. However, this is an unnecessary assumption. In the lemma below we apply the energy estimate in Section~\ref{sec:Carleman} aiming at eliminating this assumption, see also \cite{Yamamoto2016}. This lemma will be the key component in showing the main theorems.


\begin{lemma}\label{main result 2}
Let $q$ and $k$ satisfy the assumption in Lemma \ref{main result 1}, and let $u\in L^2 ([0,T];L^2(\Omega\times\R^3))$ satisfy the problem
\begin{align}\label{eqn:u_eqn_main_2}
\begin{cases}
\p_t u+v\cdot\nabla_x u+\textbf{E}\cdot\nabla_v u+q u=k & \hbox{in } (0,T)\times \Omega\times  \R^3,\\
u(0,x,v)= 0  &\hbox{in }\Omega\times \R^3.\\
\end{cases} 
\end{align}
In the pseudo 3D case, assuming $u$ and $\partial_t u\in L^2 ([0,T];L^2(\Gamma_-))$. Then one has
\begin{align}\label{k0 final estimate}
\int_{\Omega\times\R^3}   |k_0|^2  dxdv \leq 2C  \int^T_0\int_{\Gamma_+\cup\Gamma_-} |\p_tu|^2d\sigma dt \, 
\end{align}
for some constant $C>0$. Moreover, if $u$ also vanishes on $(0,T)\times\Gamma_-$, then 
\begin{align}\label{k0 stability estimate}
c \int^T_0\int_{\Gamma_+} |\p_tu|^2d\sigma dt \leq \int_{\Omega\times\R^3}   |k_0|^2  dxdv \leq C \int^T_0\int_{\Gamma_+} |\p_tu|^2d\sigma dt 
\end{align}
for some constants $c>0$ and $C>0$.
\end{lemma}
\begin{proof} 
We first take the operator $\partial_t$ on equation~\eqref{eqn:u_eqn_main_2}, then we have
\begin{align}\label{boltzmann3}
\p_t ( \p_tu)+v\cdot\nabla_x (\p_t u)+\textbf{E} \cdot\nabla_v (\p_tu)+q (\p_t u)=k_0\p_tk_1\,,
\end{align}
with the initial condition
\[
\p_t u(0,x,v)= k_0(x,v)k_1(0,x,v)\,.
\]
Applying \eqref{new Green identity} in Lemma \ref{lemma energy} to \eqref{boltzmann3}, we get
\begin{align}\label{u second est}
&\int_{\Omega\times\R^3}|\p_tu|^2 dxdv +\int^T_0\int_{\Gamma_+}|\p_tu|^2d\sigma dt \notag\\
\leq\ &  C\left(\int_{\Omega\times\R^3} |\p_tu(0,x,v)|^2 dxdv+\int^T_0\int_{\Gamma_-}|\p_tu|^2d\sigma dt+ \int^T_0\int_{\Omega\times\R^3}| k_0\p_tk_1|^2 dxdvdt \right) \notag\\
\leq\ &   C\left(\int_{\Omega\times\R^3}  |k_0|^2 dxdv+\int^T_0\int_{\Gamma_-}|\p_tu|^2d\sigma dt\right) 
\end{align}
for all $0\leq t\leq T$, where we used \eqref{k1 estimate} and combined the first and the third terms. 
In particular, from \eqref{u second est}, we have
\begin{align*}
\|\p_tu \|^2_{L^2([0,T]\times\Omega\times\R^3)}\leq  C\left(\int_{\Omega\times\R^3}  |k_0|^2 dxdv+\int^T_0\int_{\Gamma_-}|\p_tu|^2d\sigma dt\right).
\end{align*}
Using the above inequality to replace $\|\p_t u\|_{L^2([0,T]\times\Omega\times\R^3)}^2 $ in the RHS of the first inequality in \eqref{k0 estimate}, one follows a similar argument as in the proof of Lemma \ref{main result 1} to deduce that
\begin{align}\label{k_0 whole boundary}
 \int_{\Omega\times\R^3}   |k_0|^2  dxdv \leq 2Ce^{Cs} \int^T_0\int_{\Gamma_+\cup \Gamma_-}  |\p_tu|^2d\sigma dt\,.
\end{align}

Furthermore, assuming $u=0$ on $(0,T)\times \Gamma_-$, one further deduces from \eqref{u second est} that
 \begin{align*}
    \int^T_0\int_{\Gamma_+}|\p_tu|^2d\sigma dt \leq C \int_{\Omega\times\R^3}   |k_0|^2  dxdv\,,
 \end{align*}
 and from \eqref{k_0 whole boundary} that
 \begin{align}\ 
 \int_{\Omega\times\R^3}   |k_0|^2  dxdv \leq 2Ce^{Cs} \int^T_0\int_{\Gamma_+} |\p_tu|^2d\sigma dt\,.
 \end{align}
This completes the proof. 
\end{proof}

We note the main difference between Lemma \ref{main result 1} and Lemma~\ref{main result 2} lies in the fact that the latter one does not require the assumption on the boundedness of $\|\p_t u\|_{L^2([0,T]\times\Omega\times\R^3)}$, a term that we do not have a-priori knowledge about. These two lemmas allow us to finally show the two main theorems.

\subsection{Stability estimates in the reconstruction}\label{proofs main thms}
We first show the uniqueness and the stability in the reconstruction of $q$ in Theorem~\ref{main thm stability estimate}.
\begin{proof}[Proof of Theorem \ref{main thm stability estimate}]
We denote $u_{j}$ the solution to equation~\eqref{boltzmann4} with the associated $q_j$ for $j=1,2$. Let $U=u_1-u_2$. Then $U$ satisfies the equation
\begin{align} 
\p_t U+v\cdot\nabla_x U +\textbf{E}\cdot\nabla_v U+q_1 U=(q_2-q_1) u_2 
\end{align}
with trivial initial condition and trivial boundary condition
\[
U(0,x,v)= 0\,,\quad U(t,x,v)|_{(0,T)\times \Gamma_-}=0\,.
\]

Since $u_2(0,x,v) = g $, we have
\[
\p_t u_2 (0,x,v)= -v\cdot\nabla_x g  -\textbf{E} \cdot\nabla_v g - q_2 g  +S \,.
\] 
Moreover, according to the assumption in the theorem, there exist positive constants $C_1$ and $C_2$ such that $u_2\in \mathcal{P}_\mathcal{U}$, that is,
\begin{align*}
    \|u_2(0,x,v)\|_{L^\infty( \Omega \times \mathcal{U})}\leq C_1\,,\ \ \|\p_t u_2 \|_{L^\infty([0,T]\times \Omega \times \mathcal{U})}\leq C_1,\ \ \hbox{and}\ \ \inf_{ \Omega \times \mathcal{U}} |u_2(0,x,v)|\geq C_2\,.
\end{align*}
By applying Lemma \ref{main result 2}, one has 
\begin{align}
c  \int^T_0\int_{\Gamma_+} |\p_tU|^2d\sigma dt\leq \int_{\Omega\times\R^3}   |q_1-q_2 |^2  dxdv \leq C\int^T_0\int_{\Gamma_+} |\p_tU|^2d\sigma dt\,,
\end{align}
which completes the proof.
\end{proof}

To show the reconstruction of $S$, one follows the same strategy. Assuming the source has the form $S_j(t,x,v)=\widetilde S_j(x,v)S_0(t,x,v)$. Let $u_j$ be the solution to~\eqref{boltzmann diff source} with the associated $S_j$. Suppose that
$S_0(t,x,v)\in \mathcal{P}_{\mathcal{U}}$ satisfies
\begin{align}\label{sec: bounds for S}
\|S_0(0,x,v)\|_{L^\infty( \Omega \times \mathcal{U})}\leq C_1\,,\ \ \|\p_t S_0\|_{L^\infty([0,T]\times \Omega \times\mathcal{U})}\leq C_1,\ \  \hbox{and }\inf_{ \Omega \times \mathcal{U}}|S_0(0,x,v)|\geq C_2\,.
\end{align}

\begin{proof}[Proof of Theorem \ref{thm: source}]
Let $U=u_1-u_2$, then $U$ satisfies
\begin{align*}
\p_t U+v\cdot\nabla_x U +\textbf{E}\cdot\nabla_v U+qU=(\widetilde{S}_1-\widetilde{S}_2) S_0 
\end{align*}
with trivial boundary and initial data. The stabilities \eqref{diff source introduction estimate} hold by using Lemma \ref{main result 2} again.
\end{proof}

We finally comment that the reconstruction of the force $\textbf{E}$ is expected to be different from the one of $q$ or $S$. Since $\textbf{E}$ is involved in the definition of the trajectory, the weight function $\varphi$ defined in~\eqref{phase} to reconstruct $q$ or $S$ cannot be applied directly. In particular, $\Psi$ defined in \eqref{def Psi} consists the information of the force $\textbf{E}$, and it is unclear at this point how to eliminate the effects contributed from $\textbf{E}$. This issue will be investigated in the future project.

\section{Numerical Experiments}\label{sec:numerics}
In this section we present the numerical evidence of Theorem~\ref{main thm stability estimate} and Theorem~\ref{thm: source}. In particular, we will demonstrate that the $L^2$ norm of the difference between $\partial_t u_1$ and $\partial_t u_2$ indeed is proportional to the discrepancy in $q$ and in $S$.

Numerically, we choose the domain in 2D, with $(x,y)\in[0,1]^2$. Velocity space is truncated with $(v_x,v_y)\in[-6,6]^2$. In time we use simple forward Euler method, with upwinding to deal with both advection terms, $\nabla_xu$ and $\nabla_vu$. Final time is set as $T = 0.5$, and the CFL coefficient is set to be $1.2$, namely $\Delta t = \frac{\Delta x}{1.2 v_\text{max}}$, so that CFL condition is satisfied.

We emphasize in this section that we do not design some variation of PDE-constraint minimization problem or utilize the Bayesian formulation as the numerical tool for the reconstruction, but to demonstrate that the discrepancy in $q$ and in $S$ indeed gets linearly reflected in the measurements.

\subsection{Reconstructions of $q$}\label{numerics q}
To reconstruct $q$, we set the electric field to be
\begin{equation*}
\mathbf{E} = [0.3+0.1\cos{(2\pi x)}\sin{(4\pi y)}\,,0.2+0.15\sin{(2\pi x)}\cos{(4\pi y)}]^\top\,.
\end{equation*}
We compute the solution with six different absorption coefficients:
\begin{equation*}
q_\eta = \eta (0.3\sin{(2\pi x)}\cos{(4\pi y)}+0.4)\,,
\end{equation*}
where $\eta=1\,,\cdots\,,6$. Moreover, the source term $S$ is set to be $0$ and the boundary measurement $\partial_t u_\eta|_{\Gamma_+}$ is computed. We then plot the discrepancy
\[
\|\partial_t u_\eta-\partial_t u_1\|_{L^2([0,T]\times\Gamma_+)}\quad \text{and}\quad \|q_\eta-q_1\|_{L^2(\Omega\times\R^3)}\ \ \hbox{for }\eta =2,\cdots, 6\,,
\]
where we treat the case $\eta=1$ as a reference.
From Figure~\ref{fig:attenuate} we can see the discrepancy in the measurement is roughly linear fit to the discrepancy in the absorption coefficient $q$.
\begin{figure}[ht]
	\centering
	\includegraphics[width=3.0in]{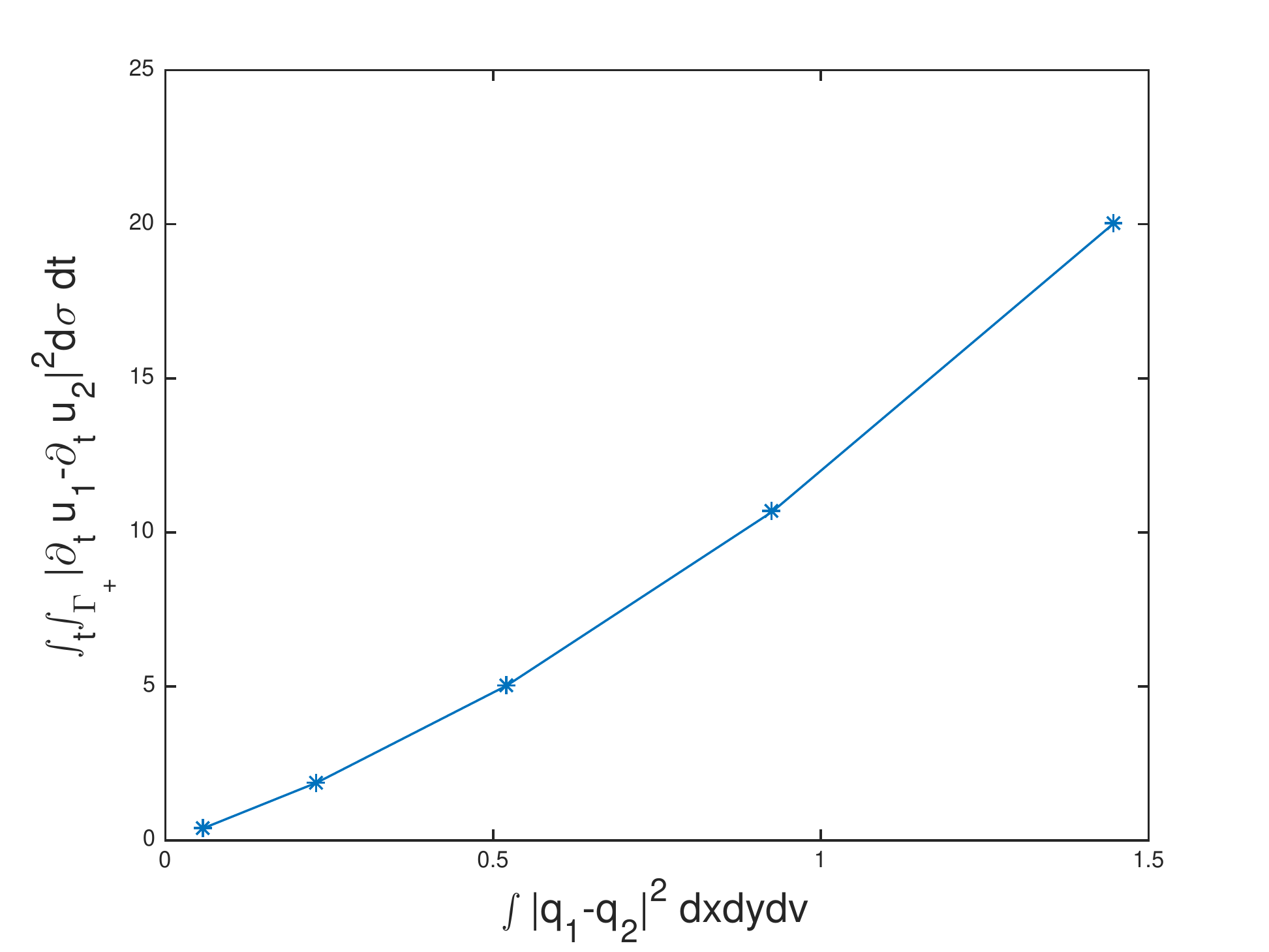} 
	\caption{\small The difference of the coefficient $\|q_\eta-q_1\|_{L^2}$ and that of the boundary data $\|\p_t u_\eta-\p_t u_1\|_{L^2}$ are calculated for $\eta=2,\cdots,6$. These $5$ points are placed as shown.}
	\label{fig:attenuate}
\end{figure}

\subsection{Reconstructions of $S$}\label{numerics S}
To reconstruct $S$, we use the same electric field
\begin{equation*}
\mathbf{E} = [0.3+0.1\cos{(2\pi x)}\sin{(4\pi y)}\,,0.2+0.15\sin{(2\pi x)}\cos{(4\pi y)}]^\top\,,
\end{equation*}
and we compute the solution with six different source terms:
\begin{equation*}
S_\eta = \eta (0.3\sin{(2\pi x)}\cos{(4\pi y)}+0.4)\,,
\end{equation*}
where $\eta=1\,,\cdots\,,6$. In addition, the absorption coefficient $q$ is set to be $0$. We compute the solution on the outgoing coordinates, that is,  $\partial_t u_\eta|_{\Gamma_+}$. Thus, we plot the discrepancy
\[
\|\partial_t u_\eta-\partial_t u_1\|_{L^2([0,T]\times\Gamma_+)}\quad \text{and}\quad \|S_\eta-S_1\|_{L^2(\Omega\times\R^3)}\ \ \hbox{for }\eta =2,\cdots, 6\,.
\]
We also observe, from Figure~\ref{fig:sourceerror}, that the discrepancy in the measurement and that in the source $S$ are almost linear.
\begin{figure}[ht]
	\centering
	\includegraphics[width=3.0in]{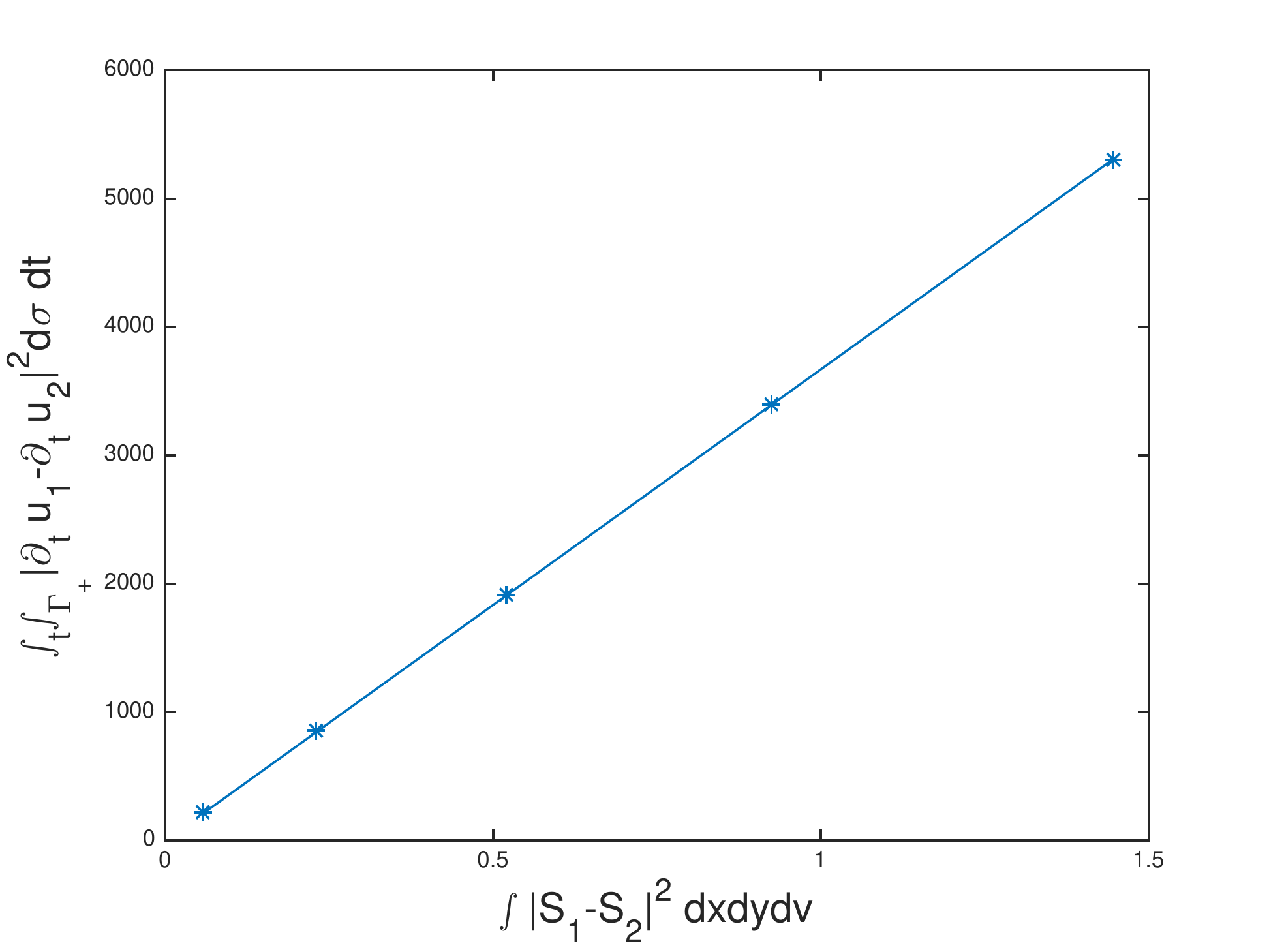} 
	\caption{\small The difference of the coefficient $\|S_\eta-S_1\|_{L^2}$ and that of the boundary data $\|\p_t u_\eta-\p_t u_1\|_{L^2}$ are calculated for $\eta=2,\cdots,6$. These $5$ points are placed as shown. }
	\label{fig:sourceerror}
\end{figure}


\vskip1cm
\noindent\textbf{Acknowledgement.}
R.-Y. Lai is partially supported by NSF grant DMS-1714490. Q. Li is partially supported by NSF grants DMS-1619778 and DMS-1750488.
Both authors would like to thank professor Chanwoo Kim for helpful discussions. The first author thanks the department of mathematics of the University of Wisconsin for the hospitality during her visit in January 2019, where part of the work was completed.

\vskip1cm
\bibliographystyle{abbrv}
\bibliography{transbib}

\end{document}